\documentclass[12pt]{amsart}
\usepackage{a4wide,amssymb}
\usepackage{tikz}  
\usepackage{color} 

\newtheorem{theorem}{Theorem}[section]

\newtheorem{lemma}[theorem]{Lemma}
\newtheorem{corollary}[theorem]{Corollary}

\begin{document}

\title{Cayley properties of merged Johnson graphs}

\author[G.A.Jones]{Gareth A.\ Jones}
\address{Gareth A.\ Jones, Mathematical Sciences, 
  \newline ${}$\hskip 20pt 
University of Southampton, Southampton, United Kingdom
\newline
 ${}$\hskip 100pt and
\newline
${}$\hskip20pt
Institute of Mathematics and Computer Science,
${}$\hskip20pt
\newline
${}$\hskip20pt
Matej Bel University, Bansk\'a Bystrica, Slovakia
}
\email{G.A.Jones@maths.soton.ac.uk}

\author[R.\ Jajcay]{Robert Jajcay}
\address{Robert Jajcay, 
  \newline ${}$\hskip 20pt 
Comenius University, Bratislava, Slovakia \\
\newline
 ${}$\hskip 100pt and
\newline  ${}$\hskip 20pt 
University of Primorska, Koper, Slovenia}
 \email{robert.jajcay@fmph.uniba.sk}

\begin{abstract} 
Extending earlier results of Godsil and of Dobson and Malni\v c on Johnson graphs, we characterise those merged Johnson graphs $J=J(n,k)_I$ which are Cayley graphs, that is, which are connected and have a group of automorphisms acting regularly on the vertices. We also characterise the merged Johnson graphs which are not Cayley graphs but which have a transitive group of automorphisms with vertex-stabilisers of order $2$. Even though these merged Johnson graphs are all vertex-transitive, we show that only relatively few of them are Cayley graphs or have a transitive group of automorphisms with vertex-stabilisers of order $2$.

\end{abstract}

\date{} 

\maketitle


20B05, 
20B20, 
20B25 
(secondary).
\medskip



\section{Introduction}\label{Intro}  

There is a folklore conjecture that, up to a given sufficiently large order, most connected vertex-transitive graphs are Cayley graphs, or equivalently, by a result of Sabidussi~\cite{Sab64}, that most such graphs have a group of automorphisms acting regularly on the vertices. Censuses of graphs of small order, such as that by McKay and Praeger in~\cite{MP94}, or that of cubic graphs by Poto\v cnik, Spiga and Verret in~\cite{PSV}, provide evidence to support this conjecture. However, if one restricts attention to certain particular families of vertex-transitive graphs, the picture can be very different. 

The smallest vertex-transitive graph which is not a Cayley graph is the Petersen graph. This is the odd graph $O_3$, or equivalently the Kneser graph $K(5,2)$. More generally, the {\em Kneser graph\/} $K(n,k)$ has the set $N\choose k$ of $k$-element subsets of an $n$-element set $N$ as its vertices, adjacent if and only if they are disjoint; the action of the symmetric group $S_n$ on $N$ shows that Kneser graphs are vertex-transitive. Taking $n=2k+1$ gives the {\em odd graph\/} $O_{k+1}:=K(2k+1,k)$. Answering a question of Biggs~\cite{Big} about whether any odd graphs are Cayley graphs, Godsil showed in~\cite{God} that $K(n,k)$ is not a Cayley graph unless either
$k=2$, $n$ is a prime power and $n\equiv 3$ mod~$(4)$, or $k=3$ and $n=8$ or $32$.

Godsil's proof uses the fact that $S_n$ is the full automorphism group of $K(n,k)$, so a group of automorphisms acts vertex-transitively on this graph if and only if it is a $k$-homogeneous permutation group of degree $n$ (one which is transitive on $k$-element subsets), and it acts regularly if and only if it is sharply $k$-homogeneous (regular on $k$-element subsets). Results of Kantor~\cite{Kan72} and of Livingstone and Wagner~\cite{LW}, summarised here in Theorem~\ref{k-homog}, provide a complete description of the latter groups (see cases (1), (2) and (3) of Theorem~\ref{MergedJohnsonCayley}), so the result follows.

The {\em Johnson graph\/} $J(n,k)$ has the same vertex set $N\choose k$ as $K(n,k)$, but in this case two $k$-element subsets are adjacent if and only if their intersection has $k-1$ elements. (We will take $k\ge 2$ to avoid considering complete graphs $J(n,1)=K_n$, while the isomorphism $J(n,k)\cong J(n,n-k)$ given by taking complements allows us to assume that $k\le n/2$.) As observed recently by Dobson and Malni\v c~\cite{DoMa}, Godsil's result extends to the Johnson graphs $J(n,k)$ with $2\le k\le (n-1)/2$, since they also have automorphism groups isomorphic to $S_n$. The same authors, as a corollary to a study of groups acting on partitions, have obtained similar results for the Johnson graphs $J(n, n/2)$ and for the folded Johnson graphs, quotients of $J(n,n/2)$ by complementation. 

The {\em distance $i$ Johnson graphs\/} $J(n,k)_i$, for $i=1, \ldots, k$, have vertex set $N\choose k$, but with pairs of vertices adjacent if and only if they are at distance $i$ in $J(n,k)$, that is, if and only if the corresponding subsets of $N$ have $k-i$ elements in common. Thus $J(n,k)=J(n,k)_1$ and $K(n,k)=J(n,k)_k$, for example. If $I$ is any non-empty subset of $\{1,\ldots, k\}$ one can form the {\em merged Johnson graph}
\[J(n,k)_I=\bigcup_{i\in I}J(n,k)_i,\]
in which two $k$-element subsets of $N$ are adjacent if and only if they are at distance $i$ in $J(n,k)$ for some $i\in I$.

The automorphism groups of the merged Johnson 
graphs $J=J(n,k)_I$ have been described in~\cite[Theorem~2]{Jon}, repeated here as Theorem~\ref{AutJ}. For most values of the parameters, ${\rm Aut}\,J=S_n$. However, in some cases, 
${\rm Aut}\,J$ is significantly larger than $S_n$ and contains $S_n$ as a proper subgroup. 
This makes it impossible to extend the results of Godsil or
of Dobson and Malni\v c mentioned above along the same lines to such graphs. In the following theorem we use different techniques to extend their results 
to all the merged Johnson graphs $J(n,k)_I$. Perhaps surprisingly, none of 
the cases where ${\rm Aut}\,J$ contains $S_n$ as a proper subgroup yields new Cayley graphs, apart from some rather trivial examples in cases (4) and (5), where the automorphism group is particularly large.

\begin{theorem}\label{MergedJohnsonCayley}
Let $2\le k\le n/2$ and let $I$ be a non-empty subset of $\{1,\ldots, k\}$. Then the merged Johnson graph $J=J(n,k)_I$ has a regular group $G$ of automorphisms if and only if one of the following holds:

\begin{enumerate}

\item $n$ is a prime power, $n\equiv 3$ {\rm mod}~$(4)$ 
and $k=2$, with any $I$, and $G\cong AHL_1(F)$ acting on some Dickson near-field $N=F$ of order $n$;

\item $n=8$ and $k=3$, with any $I$, and $G\cong AGL_1(8)$ acting on the finite field $N={\mathbb F}_8$;

\item $n=32$ and $k=3$, with any $I$, and $G\cong A\Gamma L_1(32)$ acting on the finite field $N={\mathbb F}_{32}$;

\item $I=\{1,\ldots, k\}$, with any $n$ and $k$, and $G$ any group of order $n\choose k$, acting on itself by right multiplication;

\item $k=n/2$ and $I=\{k\}$ or $\{1,\ldots, k-1\}$, with $G$ any group of order $n\choose k$, acting on itself by right multiplication.

\end{enumerate}
\end{theorem}

This result characterises those graphs $J(n,k)_I$ which are Cayley graphs, if we exclude those with $k=n/2$ and $I=\{k\}$, which are not connected.

Here, if $F$ is a Dickson near-field then $A\Gamma L_1(F)$ is the group of transformations $t\mapsto t^{\gamma}a+b$ of $F$, where $a, b\in F$, $a\ne 0$ and $\gamma$ is an automorphism of $F$, while $AGL_1(F)$ is the subgroup for which $\gamma$ is the identity. If $n$ is odd then we denote by $AHL_1(F)$ the subgroup of index $2$ in $AGL_1(F)$ for which $a$ is a non-zero square (here `$H$' stands for `half'). Each field is a near-field, and when $F$ is the finite field ${\mathbb F}_n$ of order $n$ we write simply $A\Gamma L_1(n)$, etc. (See~\cite[\S1.12]{Cam99}, \cite[\S7.6]{DM}, \cite[\S20.7]{Hal} or the Appendix for near-fields and their associated groups.)

The actions in Theorem~\ref{MergedJohnsonCayley}(1) for which $F$ is a field can be explained and constructed using maps on surfaces. The graph $J(n,2)$ is isomorphic to the line graph $L(K_n)$ of the complete graph $K_n$. A construction due to Biggs~\cite{Big71} shows that for each prime power $n\equiv 3$ mod~$(4)$, $AHL_1(n)$ acts regularly on the edges of an orientably regular map with $1$-skeleton $K_n$, inducing a regular group of automorphisms of $J(n,2)$. There are similar interpretations, in terms of maps on surfaces, for the actions in parts~(2) and (3) of Theorem~\ref{MergedJohnsonCayley}. 

Theorem~\ref{MergedJohnsonCayley} is proved in Section~\ref{Cayleyprops}, using the description of ${\rm Aut}\,J$ in Theorem~\ref{AutJ}. The cases where ${\rm Aut}\,J=S_n$ are dealt with as in Godsil's theorem, by using known results on $k$-homogeneous groups. This method also shows that if $5<k<(n-1)/2$, and in addition in most cases where $k=4, 5$ or $(n-1)/2$, the groups $A_n$ and $S_n$ are the only vertex-transitive groups of automorphisms of $J$, again extending a result of Godsil~\cite{God} for the odd graphs $O_{k+1}=J(2k+1,k)_k$; see Corollary~\ref{onlyAnSn} for details. However, those cases where ${\rm Aut}\,J > S_n$ require other methods, based in some cases, for example, on considering $G\cap S_n$ as a permutation group on $N$. 

Theorem~\ref{MergedJohnsonCayley} was motivated by a paper of Gauyacq~\cite{gau}, in which the author studies quasi-Cayley graphs. These are graphs with a regular family $\mathcal F$ of automorphisms{, meaning that for each ordered pair $(v,w)$ of vertices there is a unique automorphism in $\mathcal F$ taking $v$ to $w$. This generalises the concept of a Cayley graph, where $\mathcal F$ is a {\sl group\/} of automorphisms. In~\cite{JJ} we generalise this further to the concept of an $r$-regular family of automorphisms, where there are exactly $r$ elements of $\mathcal F$ taking $v$ to $w$. The following theorem characterises those merged Johnson graphs which have a $2$-regular group of automorphisms (by which we mean one which is transitive, with vertex-stabilisers of order $2$).  For instance, this class of graphs includes the well-known example of the Petersen graph $O_3=J(5,2)_2$ in Theorem~\ref{MergedJohnson2reg}(1), where we take $F={\mathbb F}_5$.

\begin{theorem}\label{MergedJohnson2reg}
Let $2\le k\le n/2$ and let $I$ be a non-empty subset of $\{1,\ldots, k\}$. Then $G$ is a $2$-regular group of automorphisms of the graph $J=J(n,k)_I$ if and only if one of the following holds:
\begin{enumerate}

\item $n$ is a prime power, $k=2$, with any $I$, and $G\cong AGL_1(F)$ for some near-field $N=F$ of order $n$;

\item $n=6$ and $k=3$, where $G\cong AGL_1(5)\times S_2$, with any $I$, and $AGL_1(5)$ acting naturally on the projective line $N={\mathbb P}^1({\mathbb F}_5)$, fixing $\infty$, and $S_2$ generated by complementation in $N$;

\item $n=10$, $k=5$  and $I=\{1,4\}$, $\{2,3\}$, $\{1,4,5\}$ or $\{2,3,5\}$, where $G\cong PSL_2(8)$ (this action of $G$ is explained in Section~\ref{2-reggps});

\item $n$ is even, $k=n/2$ and $I=\{k\}$ or $\{1,\ldots, k-1\}$, where $G$ is any group of order $2{n\choose k}$ with a non-normal subgroup $H$ of order $2$, acting on the cosets of $H$;

\item $I=\{1,\ldots, k\}$, where $G$ is any group of order $2{n\choose k}$ with a non-normal subgroup $H$ of order $2$, acting on the cosets of $H$.

\end{enumerate}

\end{theorem}

The structure of the proof is similar to that for Theorem~\ref{MergedJohnsonCayley}, though the details are rather different. In the cases where ${\rm Aut}\,J=S_n$, for instance, we need to consider groups which are $k$-homogeneous with stabilisers of $k$-element subsets having order $2$, rather than $1$. Proving the existence and essential uniqueness of the $2$-regular action of $PSL_2(8)$ in Theorem~\ref{MergedJohnson2reg}(3) requires non-trivial results from the cohomology of groups, such as Shapiro's Lemma; we are grateful to Ian Leary for suggesting the use of this result.

Sabidussi~\cite{Sab64} and Godsil~\cite{God}, using different but equivalent definitions, have introduced a parameter $d(\Gamma)$, called deviation or deficiency, which measures the degree to which a vertex-transitive graph $\Gamma$ differs from being a Cayley graph. We will call this the Cayley deficiency of $\Gamma$, and will use Godsil's definition of $d(\Gamma)$ as the least order of the vertex-stabilisers in any vertex-transitive group of automorphisms of $\Gamma$ (see Section~\ref{devanddef}). Thus $d(\Gamma)=1$ if and only if $\Gamma$ is a Cayley graph, and one can view Theorems~\ref{MergedJohnsonCayley} and \ref{MergedJohnson2reg} as classifying those merged Johnson graphs $J=J(n,k)_I$ which have Cayley deficiency $d(J)=1$ or $2$. By contrast, for most choices of $n, k$ and $I$, the Cayley deficiency $d(J)$ increases faster than exponentially as $n\to\infty$, so that $J$ is very far from being a Cayley graph.


\section{Johnson graphs and their automorphisms}\label{Johnsonautos}

The Johnson graphs $J(n,k)$, the distance $i$ Johnson graphs $J(n,k)_i$, the Kneser graphs $K(n,k)=J(n,k)_k$ and the merged Johnson graphs $J(n,k)_I$, defined in Section~\ref{Intro}, all have as their vertex-set the set $N\choose k$ of $k$-element subsets of an $n$-element set $N$; we will always assume that $2\le k\le n/2$, and that $\emptyset\ne I\subseteq\{1,\ldots, k\}$. We will consider all of these graphs as particular cases of the graph $J:=J(n,k)_I$, where two $k$-element subsets $K, K'\subseteq N$ are adjacent in $J$ if and only if $|K\cap K'|=k-i$ for some $i\in I$. We will generally take $N$ to be the set $\{1,\ldots, n\}$, though in some cases we may identify it with a structure such as a field or projective line with $n$ elements. 

The symmetric group $S_n$, acting naturally on $N$, preserves  the cardinalities of intersections of subsets of $N$, so it is embedded as a subgroup of ${\rm Aut}\,J$ for each $J$. In most cases $S_n$ is the full automorphism group of $J$: the exceptions are listed later in Theorem~\ref{AutJ}. In particular, ${\rm Aut}\,J=S_n$ in all cases where $2\le k<(n-1)/2$, with the exception of the complementary graphs $J(12,4)_I$ for $I=\{1,3\}$ and $\{2,4\}$. For this reason, we will first study groups $G$ of automorphisms of graphs $J=J(n,k)_I$ arising from subgroups of $S_n$.


\section{$k$-homogeneous groups}\label{k-homoggps}

A subgroup $G$ of $S_n$ is vertex-transitive on the graph $J=J(n,k)_I$ if and only if it acts transitively on $N\choose k$, that is, $G$ is a $k$-homogeneous permutation group. Such groups have been classified, and we list them for each $k\ge 2$ later in this section. (We omit the case $k=1$, since $1$-homogeneity is equivalent to transitivity, and $ J(n,1)_1$ is isomorphic to the complement of $K_n$.)

First of all, it is clear that any $k$-transitive group is $k$-homogeneous. For each $k\ge 2$ the $k$-transitive finite groups are known as a result of the classification of finite simple groups. Apart from the alternating and symmetric groups $A_n$ and $S_n$, the multiply transitive finite groups $G$ form six infinite families and one finite family, consisting of the following:
\begin{itemize}
\item affine groups, that is, various subgroups $G\le A\Gamma L_d(q)$, of degree $n=q^d$ for some prime power $q$ and $d\ge 1$;
\item projective groups, of degree $n=(q^d-1)/(q-1)$ for some prime power $q$ and $d\ge 2$, such that $PSL_d(q)\le G\le P\Gamma L_d(q)$;
\item unitary groups, of degree $n=q^3+1$ for some prime power $q$, such that $PSU_3(q)\le G\le P\Gamma U_3(q)$;
\item symplectic groups $G=Sp_{2d}(2)$, each with two representations of degrees $n=2^{d-1}(2^d\pm 1)$, for some $d\ge 3$;
\item Suzuki groups, of degree $n=q^2+1$ for some $q=2^e$ with odd $e\ge 3$, such that $Sz(q)\le G\le {\rm Aut}\,Sz(q)$;
\item Ree groups, of degree $n=q^3+1$ for some $q=3^e$ with odd $e\ge 3$, such that $Re(q)\le G\le {\rm Aut}\,Re(q)$;
\item a finite number of sporadic examples: the Mathieu groups $M_n$ for $n=11, 12, 22, 23$ and $24$, and ${\rm Aut}\,M_{22}$ for $n=22$, acting on Steiner systems; $PSL_2(11)$ acting on the cosets of $A_5$ for $n=11$; $M_{11}$ on the cosets of $PSL_2(11)$ for $n=12$; $A_7$ on the cosets of $AGL_3(2)$ for $n=15$; $P\Sigma L_2(8)\cong Re(3)$ on its Sylow $3$-subgroups for $n=28$; the Higman-Sims group $HS$ for $n=176$; the Conway group $Co_3$ for $n=276$.
\end{itemize}

(Apart from the affine groups, the groups listed here are all finite simple or almost simple groups; see~\cite{ATLAS} or~\cite{Wil} for the definitions and general properties of these groups. For a detailed description of them as multiply transitive permutation groups, see~\cite[\S7.7]{DM}, though it omits ${\rm Aut}\,M_{22}$, or~\cite[Chapter~7]{Cam99} for a more concise but complete list.) The degrees of the groups in these seven families form a subset of $\mathbb N$ of asymptotic density~$0$; see~\cite{CNT} for a more precise estimate of the density of this set.

None of the groups in this list is $6$-transitive, and the only $5$-transitive groups are $M_{12}$ and $M_{24}$. In addition to these, the only $4$-transitive groups are $M_{11}$ and $M_{23}$, and the only $3$-transitive groups are various subgroups of $AGL_d(2)$ and $P\Gamma L_2(q)$, together with $M_{11}$ (for $n=12$), $M_{22}$ and ${\rm Aut}\,M_{22}$. All the remaining groups are merely $2$-transitive.

In most cases the converse is also true, that is, a $k$-homogenous group is in fact $k$-transitive. The following theorem, based on results of Livingstone and Wagner~\cite{LW} and Kantor~\cite{Kan72}, shows where the exceptions occur:

\begin{theorem}\label{k-homog}
Let $G$ be a $k$-homogeneous permutation group of finite degree $n$, where $2\le k\le n/2$. Then $G$ is $(k-1)$-transitive. Moreover, $G$ is $k$-transitive unless either:
\begin{enumerate}
\item $k=2$ and $G\le A\Gamma L_1(q)$ for some prime power $n=q\equiv 3$ {\rm mod}~$(4)$, or
\item $k=3$ and $PSL_2(q)\le G\le P\Gamma L_2(q)$ for some prime power $q\equiv 3$ {\rm mod}~$(4)$, with $n=q+1$, or
\item $k=3$ and $G=AGL_1(8)$, $A\Gamma L_1(8)$ or $A\Gamma L_1(32)$ with $n=8, 8$ or $32$, or
\item $k=4$ and $G=PSL_2(8)$, $P\Gamma L_2(8)$ or $P\Gamma L_2(32)$ with $n=9, 9$ or $33$.
\end{enumerate}
In each case, the permutation representation of $G$ is the natural one. 
\end{theorem}

(Note that each of the groups $G$ in cases~(3) and (4) is $k$-homogeneous but not $k$-transitive, whereas this applies to only some of the groups described in cases (1) and (2): for instance $A\Gamma L_1(q)$ and $P\Gamma L_2(q)$ are $k$-transitive for $k=2$ and $3$ respectively.)

It will be useful to consider, from among the multiply homogeneous groups listed in this section, the set ${\mathcal H}_k$ of $k$-homogeneous groups, other than $S_n$ and $A_n$, for each $k\ge 2$, together with the set ${\mathcal D}_k$ of their degrees. Since ${\mathcal H}_{k+1}\subseteq{\mathcal H}_k$ for each $k$, it is sufficient to describe the sets ${\mathcal H}_k\setminus{\mathcal H}_{k+1}$. We have
\begin{itemize}
\item ${\mathcal H}_k=\emptyset$ for all $k\ge 6$;
\item ${\mathcal H}_5=\{M_{12}, M_{24}\}$\;
\item ${\mathcal H}_4\setminus{\mathcal H}_5=\{M_{11}, M_{23}, PSL_2(8), P\Gamma L_2(8), P\Gamma L_2(32)\}$;
\item ${\mathcal H}_3\setminus{\mathcal H}_4$ consists of various subgroups of $A\Gamma L_d(2)$ and $P\Gamma L_2(q)$, and also $M_{11}$ (for $n=12$), $M_{22}$, ${\rm Aut}\,M_{22}$, $AGL_1(8)$, $A\Gamma L_1(8)$ and $A\Gamma L_1(32)$;
\item ${\mathcal H}_2\setminus{\mathcal H}_3$ consists of the remaining groups listed above.
\end{itemize}

By considering the degrees of these groups, stated above, we see that
\begin{itemize}
\item ${\mathcal D}_k=\emptyset$ for all $k\ge 6$;
\item ${\mathcal D}_5=\{12, 24\}$;
\item ${\mathcal D}_4=\{9, 11, 12, 23, 24, 33\}$; 
\item ${\mathcal D}_3$ consists of the integers $2^d$ where $d\ge 3$, and $q+1$ where $q$ is a prime power, together with $11, 22$ and $23$;
\item ${\mathcal D}_2$ consists of all the degrees $n$ given in the list of multiply transitive groups at the start of this section (the groups in Theorem~\ref{k-homog} contribute no further degrees).
\end{itemize}

This shows that if $k\ge 6$, or if $2\le k\le 5$ and $n$ avoids an easily described subset of $\mathbb N$ of asymptotic density $0$, then the only $k$-homogeneous groups of degree $n$ are $S_n$ and $A_n$. In fact, $k$-homogeneous groups for $k\ge 2$ are primitive, and Cameron, Neumann and Teague~\cite{CNT} have shown that for all but a set of integers $n$ of asymptotic density $0$ the only primitive groups of degree $n$ are $S_n$ and $A_n$.


\section{$r$-regular groups from $k$-homogeneous groups}\label{r-reggps}

Motivated by~\cite{gau} and~\cite{JJ}, we can now apply the information in the preceding section to obtain $r$-regular groups of automorphisms of various merged Johnson graphs from $k$-homogeneous subgroups of $S_n$. If $G$ is any vertex-transitive group of automorphisms of $J=J(n,k)_I$, it forms an $r$-regular family for $J$ where $r$ is the order $|G|/{n\choose k}$ of the vertex-stabilisers. In particular, if $G$ is a $k$-homogeneous subgroup of $S_n$ then it acts on $J$ as an $r$-regular group of automorphisms, where $r$ is the order of the subgroup of $G$ stabilising a $k$-element subset $K\subseteq N$.

For us, concerned with Cayley graphs, the case $r=1$ is the most important. The corresponding subgroups $G$ of $S_n$ are the {\sl sharply $k$-homogeneous\/} groups, those acting regularly on $k$-element subsets, so if such a group $G$ exists then $J(n,k)_I$ is a Cayley graph. Since $k\ge 2$, $G$ cannot be $k$-transitive, so it must be one of the exceptional groups described in Theorem~\ref{k-homog}. Case~(1), where $k=2$, has been dealt with by Kantor in~\cite{Kan69}. Case~(2), with $k=3$, cannot arise, since $G$ contains the element $t\mapsto 1/(1-t)$ of $PSL_2(q)$, which preserves the subset $\{0, 1, \infty\}$. Cases~(3) and (4) are easily dealt with by comparing $|G|$ with $n\choose k$. As a result, we have the following well-known easy consequence of~\cite[Theorem 1]{Kan72}:

\begin{corollary}\label{sharpk-homog} 
Let $G$ be a permutation group $G$ of finite degree $n$. Then $G$ is sharply $k$-homogeneous for some $k\ge 2$ if and only if either
\begin{enumerate}
\item $k=2$ and $G\cong AHL_1(F)$ for some Dickson near-field $F$ of prime power order $n\equiv 3$ {\rm mod}~$(4)$, or
\item $k=3$ and $G\cong AGL_1(8)$ or $A\Gamma L_1(32)$.
\end{enumerate}
\end{corollary}

It follows that $J(n,k)_I$  is a Cayley graph in each of these cases, for any non-empty set $I$.

The case $r=2$ is also of interest. For later use, in proving Theorem~\ref{MergedJohnson2reg} and in~\cite{JJ}, we also note that if $F$ is a near-field of (necessarily prime power) order $n$ then the group $G=AGL_1(F)$ is sharply $2$-transitive on $F$, so it acts $2$-regularly on $J(n,2)_I$ for any $I$.


\section{Automorphism groups of merged Johnson graphs}\label{autogps}

The above examples all arise from subgroups of $S_n$, which has an induced action on each merged Johnson graph $J=J(n,k)_I$. However, the following result~\cite[Theorem~2]{Jon} shows that in some cases ${\rm Aut}\,J$ is strictly larger than $S_n$, leading to the possibility of further vertex-transitive groups of automorphisms. First we need some notation: define $e={n\choose k}/2$, and for any $I\subseteq\{1,\ldots, k\}$ define $I'=I\setminus\{k\}$ and $I''=k-I'=\{k-i\mid i\in I'\}$. We will assume that $I\ne\{1,\ldots, k\}$, since otherwise $J(n,k)_I$ is a complete graph and ${\rm Aut}\,J$ is the symmetric group on $N\choose k$.

\begin{theorem}\cite[Theorem~2]{Jon}\label{AutJ}
 Let $J=J(n,k)_I$ where $2\leq k\leq n/2$ and $I$ is a non-empty proper subset of $\{1,\ldots, k\}$.

\begin{enumerate}

\item If $2\leq k<(n-1)/2$, and $J\neq J(12,4)_I$ with $I=\{1,3\}$ or $\{2,4\}$, then ${\rm Aut}\,J=S_n$.

\item If $(n,k)=(12,4)$ with $I=\{1,3\}$ or $\{2,4\}$, then ${\rm Aut}\,J=GO^{-}_{10}(2)$. 

\item If $k=(n-1)/2$ and $I\neq k+1-I$, then ${\rm Aut}\,J=S_n$.

\item If $k=(n-1)/2$ and $I=k+1-I$, then ${\rm Aut}\,J=S_{n+1}$.

\item If $k=n/2$ and $I\neq\{k\}$ or $\{1,\ldots,k-1\}$, and $I'\neq I''$, then
${\rm Aut}\,J=S_2\times S_n$.

\item If $k=n/2$ and $I\neq\{k\}$ or $\{1,\ldots,k-1\}$, and $I'=I''$, then
${\rm Aut}\,J=S_2^e\rtimes S_n$.

\item If $k=n/2$ and $I=\{k\}$ or $\{1,\ldots,k-1\}$, then ${\rm Aut}\,J=S_2^e\rtimes S_e=S_2\wr
S_e$.

\end{enumerate}
\end{theorem}

In all cases, ${\rm Aut}\,J$ contains $S_n$ with its induced action on $k$-element subsets. In cases (1) and (3) this is the full automorphism group, whereas in the remaining cases, namely (2) and (4) to (7), the automorphism groups are larger; these are described in Section~\ref{Cayleyprops}, in the proof of Theorem~\ref{MergedJohnsonCayley}, where we first need to use them.

\begin{corollary}\label{onlyAnSn}
 Let $J=J(n,k)_I$ where $I$ is a non-empty proper subset of $\{1,\ldots, k\}$. Suppose that either
 \begin{enumerate}
 \item $5<k<(n-1)/2$, or
  \item $5<k=(n-1)/2$ and $I\neq k+1-I$, or
 \item $5=k<(n-1)/2$ and $n\ne 12, 24$, or
 \item $4=k<(n-1)/2$ and $n\ne 9, 11, 12, 23, 24, 33$.

 \end{enumerate}
 Then the only vertex-transitive groups of automorphisms of $J$ are $A_n$ and $S_n$, each acting naturally.
\end{corollary}

\begin{proof}
Theorem~\ref{AutJ} ensures that ${\rm Aut}\,J=S_n$ in all these cases, so that any vertex-transitive group of automorphisms of $J$ must be a $k$-homogeneous subgroup of $S_n$. As shown in Section~\ref{k-homoggps}, for $k>5$ the only such groups are $A_n$ and $S_n$. The same conclusion applies for $k=5$ provided $n\ne 12, 24$, so that the $5$-homogeneous Mathieu groups of those degrees are avoided. Similarly, one can take $k=4$ provided $n\ne 9, 11, 12, 23, 24, 33$, so that case~(2) of Theorem~\ref{AutJ} and the $4$-homogeneous groups in ${\mathcal H}_4$ are avoided.
\end{proof}

Similar results apply for $k=2$ and $k=3$, but in each case one has to exclude an infinite set ${\mathcal H}_k$ of values of $n$, of asymptotic density $0$.


\section{Deviation and deficiency}\label{devanddef}

As mentioned in the Introduction, Cayley graphs are believed to 
constitute a significant part of the class of vertex-transitive graphs. 
By definition, every vertex-transitive graph admits a vertex-transitive
group of automorphisms, which may or may not be its full automorphism 
group. By a well-known result of Sabidussi~\cite{Sab64}, 
each Cayley graph $\Gamma$ admits a vertex-transitive group of automorphisms $G$ of order $|G|=|V(\Gamma)|$, which again may or may not be the full automorphism group. 
Thus the ratio between the order of a smallest vertex-transitive subgroup of ${\rm Aut}\,\Gamma$ (rather than that of ${\rm Aut}\,\Gamma$ itself) and the order of the graph $\Gamma$ appears to be a good indicator of how far $\Gamma$ is from being a Cayley graph. 

Other such indicators have also been proposed. In~\cite{Sab64} Sabidussi defined the {\em deviation\/} ${\rm dev}(\Gamma)$ of a vertex-transitive graph $\Gamma$ to be the smallest integer $r$ such that the 
lexicographic product }$\Gamma[\overline{K_r}]$ is a Cayley graph. (Here 
$\overline{K}_r$  is the null graph with $r$ vertices and no edges, 
and the lexicographic product $\Gamma[\overline{K_r}]$ is a covering of $\Gamma$ consisting of $|V(\Gamma)|$ disjoint copies of $\overline{K_r}$, each forming the fibre over a vertex of $\Gamma$,
with two such fibres joined to make a complete bipartite graph $K_{r,r}$
if and only if the vertices they cover are adjacent in $\Gamma$.) In particular $\Gamma[\overline{K_1}]\cong\Gamma$, so ${\rm dev}(\Gamma)=1$ if and only if $\Gamma$ is a Cayley graph.

Godsil~\cite{God} defined the {\em deficiency\/} $d(\Gamma)$ of a vertex-transitive graph $\Gamma$ to be the least order $|G_v|$ of the vertex-stabilisers $G_v$ ($v\in V(\Gamma)$) in any vertex-transitive group $G$ of automorphisms of $\Gamma$. Again, $d(\Gamma)=1$ if and only if $\Gamma$ is a Cayley graph.

In \cite{JMM}, Malni\v c, Maru\v si\v c and the second author showed that
if $G$ is any vertex-transitive group of automorphisms of a graph $\Gamma$, then the Cayley graph of $G$, with the connecting set consisting of the elements of $G$ mapping a specific vertex $v$ of $\Gamma$ to its neighbors, is isomorphic to $ \Gamma(\overline{K_r}) $ where $r=|G_v|$. 
Thus Sabidussi's deviation ${\rm dev}(\Gamma)$ is equal to Godsil's deficiency $d(\Gamma)$, and since $|G_v|=|G|/|V(\Gamma)|$,  it is also equal to the ratio of the order of a smallest vertex-transitive 
group of automorphisms and the order of the graph --- the parameter we discussed at the beginning of this section. Thus all three 
are actually the same parameter, measuring how much a
vertex-transitive graph deviates from being a Cayley graph. In this paper, we will call this parameter
the {\em Cayley deficiency\/} $d(\Gamma)$. The bigger the Cayley deficiency, the 
further the graph is from being a Cayley graph, with Cayley graphs being those of
Cayley deficiency $1$.

For example, if $\Gamma$ is one of the merged Johnson graphs $J=J(n,k)_I$ listed in Corollary~\ref{onlyAnSn}, the smallest vertex-transitive 
subgroup of $ {\rm Aut}\, J $ is $A_n$, so
$d(J)=n!/2{n \choose k} = k!(n-k)!/2 $.
If $n\to\infty$ with $k$ fixed, then $d(J)$ grows faster than exponentially, while the order $n\choose k$ of $J$ has polynomial growth, of degree $k$. Thus if $n\gg k > 5 $ then $J$ is very far from being a Cayley graph.

On the other hand, even though the automorphism group of
a merged Johnson graph $J(n,k)_I$ always contains $S_n$, in cases
other than those listed in Corollary~\ref{onlyAnSn} the group $A_n$ need not be a smallest 
vertex-transitive subgroup of ${\rm Aut}\,J$. For instance the $r$-regular groups of automorphisms arising from $k$-homogeneous permutation groups, discussed in Section~\ref{r-reggps}, give examples of this with relatively small values of $r$, and hence of $d(J)$ since $d(J)\le r$. The aim of the rest of this paper is to classify the merged Johnson graphs $J$ with the smallest Cayley deficiencies, namely those with $d(J)=1$ or $2$.


\section{Cayley properties of merged Johnson graphs}\label{Cayleyprops}

We are now ready to prove Theorem~\ref{MergedJohnsonCayley}, which classifies the merged Johnson graphs of Cayley deficiency $d(J)=1$.

\noindent{\sl Proof.} We first deal with two trivial cases. If $I=\{1,\ldots, k\}$, as in Theorem~\ref{MergedJohnsonCayley}(4), then the graph $J:=J(n,k)_I$ is a complete graph on $n\choose k$ vertices; the groups $G$ which can then act regularly as a group of automorphisms of $J$ are those of order $n\choose k$, so that $J$ is a Cayley graph for $G$ with respect to the generating set $G\setminus\{1\}$.

Similarly, if $k=n/2$ and $I=\{1, \ldots, k-1\}$, as in conclusion~(5), then $J$ is this same complete graph minus a complete matching, given by complementation; in this case $G$ can act as a regular group of automorphisms of $J$ if and only if it has order $n\choose k$, so that $J$ is a Cayley graph for $G$ with respect to the generating set $G\setminus T$ for some subgroup $T$ of order $2$ (which exists since $n\choose k$ is even). The same regular groups arise if $k=n/2$ and $I=\{k\}$, giving the complementary graph $J$, the disjoint union of $e={n\choose k}/2$ copies of $K_2$; however, this graph is not connected, so it is not a Cayley graph. Having dealt with these cases, we will assume for the rest of this proof that $I\ne\{1,\ldots, k\}$, and that if $k=n/2$ then $I$ is neither $\{k\}$ nor $\{1,\ldots, k-1\}$.

In each of cases~(1), (2) and (3), Corollary~\ref{sharpk-homog} shows that the specified group $G$ acts sharply $k$-homogeneously on an $n$-element set $N$, inducing a regular group of automorphisms of $J$. 

It remains to prove the converse, that these are the only cases in which $J$ is a Cayley graph for a group $G$. We will do this by a case-by-case analysis of the various possibilities for ${\rm Aut}\,J$ (and hence for $G$), as classified in Theorem~\ref{AutJ}.

(a) First suppose that $G$ is contained in $S_n$, the symmetric group on $N$. (By Theorem~\ref{AutJ} this includes all cases in which $k<(n-1)/2$ except $J=J(12,4)_I$ with $I=\{1,3\}$ or $\{2,4\}$.) Then $G$ acts on $N$ as a sharply $k$-homogeneous group, so Corollary~\ref{sharpk-homog} implies that $G$ and $J$ satisfy conclusion~(1), (2) or (3). (This is essentially the argument used by Godsil~\cite{God} to deal with the Kneser graphs $K(n,k)$; as noted by Dobson and Malni\v c~\cite{DoMa} it also applies to the Johnson graphs $J(n,k)$ with $k<n/2$.)

(b) Next we deal with the exceptional graphs $J=J(12,4)_I$, with $I=\{1,3\}$ or $\{2,4\}$. By Theorem~\ref{AutJ} these two complementary graphs have as their automorphism group the general orthogonal group $GO^-_{10}(2)$ in ATLAS notation~\cite{ATLAS}, so it is sufficient to show that this group has no subgroup $G$ of order ${12\choose 4}=495=3^2.5.11$. Having odd order, any such group must be contained in the simple subgroup $S=O^-_{10}(2)$ of index $2$ in $GO^-_{10}(2)$.

If $n_p$ denotes the number of Sylow $p$-subgroups of $G$, then Sylow's theorems give $n_3=1$ or $55$, $n_5=1$ or $11$, and $n_{11}=1$ or $45$. If $n_5=1$ then $G$ has a normal Sylow $5$-subgroup $K\cong C_5$; since $|{\rm Aut}\,K|=4$, elements of order $11$ must centralise $K$, whereas the ATLAS shows that $S$ has no elements of order $55$. Thus $n_5=11$.

If $n_{11}=45$ there are $45.10=450$ elements of order $11$ in $G$; since there are also $11.4=44$ elements of order $5$, and one of order $1$, there are no elements of order $3$, contradicting Cauchy's Theorem. Hence $n_{11}=1$, so $G$ has a normal Sylow $11$-subgroup $E\cong C_{11}$.

A Sylow $3$-subgroup $T$ of $G$ has order $9$, coprime to $|{\rm Aut}\,E|=10$, so it centralises $E$. However, according to the ATLAS, $|C_S(g)|=33$ for each element $g\in S$ of order $11$, a contradiction.\ Thus ${\rm Aut}\,J$ contains no subgroup $G$ of order $495$, so $J$£ is not a Cayley graph. 

Parts (a) and (b) deal with all cases where $k<(n-1)/2$. To make further progress we need a technical lemma. In order to maintain continuity we postpone its proof to the end of this section.

\begin{lemma}\label{regorbits}
Let $n\ge 4$ and $k=n/2$. Then a subgroup $H\le S_n$ has two orbits on $k$-element subsets of $N=\{1,\ldots, n\}$, each of them regular, if and only if $n=4$ and $H\cong C_3$.
\end{lemma}

(c) Now suppose that $n$ is even and $k=n/2$.  We have already dealt with the cases $I=\{1,\ldots, k\}$, $\{1,\ldots, k-1\}$ and $\{k\}$, so we may assume that $n>4$. If $ n>4$ and
$k=n/2$, Theorem~\ref{AutJ} shows that ${\rm Aut}\,J$ is a direct or semidirect product $S_n\times S_2$ or $S_2^e\rtimes S_n$ as $I'\ne I''$ or $I'=I''$ (recall that $e={n\choose k}/2$, $I'=I\setminus\{k\}$ and $I''=k-I'$). 

First, suppose that $I'\ne I''$, so that ${\rm Aut}\,J=S_n\times S_2$, with $S_n$ acting naturally and $S_2$ generated by the automorphism sending each $k$-element subset to its complement. Part~(a) of this proof shows that there are no regular subgroups $G\le S_n$, so any regular subgroup $G$ must have a subgroup $H=G\cap S_n$ of index $2$, with two regular orbits on $k$-element subsets of $N$. Lemma~\ref{regorbits} then shows that $n=4$, against our assumptions, so this possibility is eliminated.

Now suppose that $I'=I''$, so that ${\rm Aut}\,J=S_2^e\rtimes S_n$; in this case the complement $S_n$ in this semidirect product acts naturally, while each direct factor $S_2$ of the normal subgroup $S_2^e$ is generated by an involution transposing one $k$-element subset with its complement, and leaving all others invariant. It follows that the only elements of $S_2^e$ acting on $J$ without fixed points are those in the diagonal subgroup, transposing every $k$-element subset with its complement. Thus $|G\cap S_2^e|\le 2$, and we can apply the argument of the preceding paragraph to the image of $G$ in $S_n$, with the same conclusion.

(d) Finally, let $n$ be odd and $k=(n-1)/2$, so that ${\rm Aut}\,J=S_n$ or $S_{n+1}$ where $I\ne k+1-I$ or $I=k+1-I$ respectively. In the first case part~(a) of this proof applies, so we may assume that $I=k+1-I$ and ${\rm Aut}\,J=S_{n+1}$.

In this case, $S_{n+1}$ acts on $J$ as follows. Let $N^*=N\cup\{n+1\}=\{1,\ldots, n+1\}$, and let $\Phi$ denote the set of all equipartitions of $N^*$, meaning the unordered pairs of complementary subsets $P_1, P_2$ of $N^*$ with $|P_i|=(n+1)/2$ for each $i=1, 2$. There is a bijection $\beta:{N\choose k}\to\Phi$, sending each $K\in{N\choose k}$ to the equipartition $\{K\cup\{n+1\}, N\setminus K\}$; its inverse sends each $\{P_1,P_2\}\in\Phi$ to $P_i\setminus\{n+1\}$ where $n+1\in P_i$. The natural action of $S_{n+1}$ on $N^*$ and hence on $\Phi$ induces, via $\beta$, an action on ${N\choose k}$ preserving $J$; its restriction to the subgroup fixing $n+1$ is the natural action of $S_n$ on the vertex set $N=\{1, 2, \ldots, n\}$ of $J$.

We need to show that ${\rm Aut}\,J$ has no subgroup $G$ acting regularly on ${N\choose k}$, or equivalently that $S_{n+1}$ has no subgroup $H$ acting regularly on $\Phi$. Such a subgroup would be a group of degree $n+1$ on $N^*$ with two regular orbits on $(k+1)$-element subsets, where $k+1=(n+1)/2$; Lemma~\ref{regorbits}, with a slight change of notation, then implies would imply that $n=3$ and so $k=1$, against our assumption that $k\ge 2$.\hfill$\square$

\smallskip

\noindent{\bf Remark.} It is straightforward to check that in each of conclusions~(2) and (3) of Theorem~\ref{MergedJohnsonCayley} there is a unique conjugacy class of regular subgroups $G$ in ${\rm Aut}\,J$, while in~$(1)$ the conjugacy classes correspond to the isomorphism classes of Dickson near-fields $F$ of order $n$; see more on this in the Appendix.

\smallskip

Lemma~\ref{regorbits}, used in the above proof, is the particular case $r=1$ of the following lemma, the other cases of which we will need later. 

\begin{lemma}\label{rregorbits}
Let $n=2k\ge 4$. Then a subgroup $H\le S_n$ has two orbits on $k$-element subsets of $N=\{1,\ldots, n\}$, both of them $r$-regular for some $r\le 4$, if and only if
\begin{enumerate}
\item $r=1$, $n=4$ and $H\cong A_3$, or
\item $r=2$, $n=4$ and $H\cong S_3$, or
\item $r=2$, $n=6$ and $H\cong AGL_1(5)$, or
\item $r=4$, $n=10$ and $H\cong PSL_2(8)$,
\end{enumerate}   
with $H$ fixing a point and having its natural transitive action on the remaining $n-1$ points. The two orbits on $k$-element subsets then consist of those subsets containing or not containing the fixed point.
\end{lemma}

\noindent{\sl Proof.} Livingstone and Wagner~\cite{LW} have shown that if $s\le t$ and $s+t\le n$, then a group of degree $n$ has at least as many orbits on $t$-element subsets as it has on $s$-element subsets (see also Cameron~\cite[Theorem~2.2]{Cam76}). It follows that a group $H$ satisfying the hypotheses of this lemma must have at most two orbits on $m$-element subsets, for each $m=1,\ldots, k$.

If $H$ is intransitive on $N$, it has two orbits $N_1$ and $N_2$, and we may assume that $|N_1|\ge|N_2|$. If $|N_2|\ge 2$ then since $|N_1|\ge k\ge 2$, $H$ must have at least three orbits on $k$-element subsets, each consisting of such subsets $K$ with $|K\cap N_2|=0, 1$ or $2$. This contradicts our hypothesis, so $|N_2|=1$ and the two orbits on $k$-element subsets consist of those $K$ containing or not containing $N_2$. 

This shows that $H$ is a $(k-1)$-homogeneous group of odd degree $n-1=2(k-1)+1$ on $N_1$. It follows that $H$ is set-transitive on $N_1$, by which we mean that $H$ is $m$-homogeneous for all $m\le n-1$. Set-transitive groups have been classified by Beaumont and Peterson~\cite{BP}. Excluding groups of even degree (alternating and symmetric groups, and $PGL_2(5)$ of degree $6$), the only possibilities for $H$ are the following, with $r$ denoting the size of the stabiliser of a $(k-1)$-element subset of $N_1$:
\begin{itemize}
\item $A_{n-1}$ and $S_{n-1}$ for even $n\ge 4$, with $r=(k-1)!k!/2$ and $(k-1)!k!$;
\item $AGL_1(5)$ of degree $n-1=5$, with $r=2$;
\item $PSL_2(8)$ and $P\Gamma L_2(8)$ of degree $n-1=9$, with $r=4$ and $12$. 
\end{itemize}
Only $A_3$, $S_3$, $AGL_1(5)$ and $PSL_2(8)$ satisfy $r\le 4$, giving conclusions (1) to (4). We may therefore assume that $H$ is transitive on $N$. 

Now suppose that $H$ is $2$-homogeneous on $N$, or equivalently, since it has even order, $2$-transitive. (This always applies when $n=4$ or $6$, since $|H|$ is divisible by ${n\choose k}/2=3$ or $10$.) Inspection of the $2$-transitive groups $H$ of even degree $n$ listed in Section~6 shows that they never have two orbits, both of size  ${n\choose k}/2$, on $k$-element subsets for $k=n/2$: in most cases, either $|H|$ is not divisible by ${n\choose k}/2$, or else $H$ preserves some geometric or combinatorial structure on $N$ for which there are at least three non-isomorphic $k$-element subsets. 

If $H$ is not $2$-homogeneous it has two orbits on pairs, and by taking one of them as the edge set we see that $H$ is a group of automorphisms of a vertex- and edge-transitive graph $\Gamma$ on $N$ of valency $v<k$. Since $n\ge 8$ the following lemma implies that $H$ has at least three orbits on $m$-element subsets for $m=3$ or $4$, giving a contradiction. \hfill$\square$

In the following lemma, we use the notation $\Gamma_1+\Gamma_2$ to denote the disjoint union of graphs $\Gamma_1$ and $\Gamma_2$, and $m\Gamma$ to denote the disjoint union of $m$ copies of a graph $\Gamma$.

\begin{lemma}\label{iso-class}
Let $\Gamma$ be a regular graph of order $n=2k\ge 8$ and valency $v$ such that $1\le v<k$. Then
\begin{enumerate}
\item if $\Gamma\cong 2K_k$ or $kK_2$ there are just two isomorphism classes of induced $3$-vertex subgraphs of $\Gamma$, and three isomorphism classes of induced $4$-vertex subgraphs;
\item  otherwise there are at least three isomorphism classes of induced $3$-vertex subgraphs.
\end{enumerate}
\end{lemma}

[Note that there are, up to isomorphism, just four simple graphs on three vertices, namely $K_3$, $K_2+K_1$ and their complements.]

\smallskip

\noindent{\sl Proof.} If $v=1$ then $\Gamma\cong kK_2$ with $k\ge 4$, so the $3$-vertex subgraphs are isomorphic to $K_2+K_1$ and $3K_1$, and the $4$-vertex subgraphs are isomorphic to $2K_2$, $K_2+2K_1$ and $4K_1$. We may therefore assume that $v\ge 2$.

First suppose that there are no induced subgraphs isomorphic to $P_3$, 
a path of three vertices. Then each vertex $a\in\Gamma$, together with its neighbours, forms a complete graph $K_{v+1}$. It follows that $\Gamma\cong mK_{v+1}$ where
$m=n/(v+1)\ge n/k\ge 2$. If $m=2$, so that $\Gamma\cong 2K_k$ with $k\ge 4$, the $3$-vertex subgraphs are isomorphic to $K_3$ or $K_2+K_1$, while the $4$-vertex subgraphs are isomorphic to $K_4$, $K_3+K_1$ or $2K_2$. If $m\ge 3$ the $3$-vertex subgraphs are isomorphic to $K_3$, $K_2+K_1$ or $3K_1$ and $\Gamma $ falls under case (2). We may therefore assume that there are induced subgraphs isomorphic to $P_3$.

Now suppose that there are no induced subgraphs isomorphic to the null graph $3K_1$. Then for any non-adjacent vertices $a$ and $b$, each remaining vertex is adjacent to $a$ or $b$, so $n\le 2+2v=2(v+1)\le 2k$. But $n=2k$, so $\Gamma(a)$ and $\Gamma(b)$ have disjoint vertex-sets, and thus $\Gamma(a)\cup\{a\}$ and $\Gamma(b)\cup\{b\}$ partition $\Gamma$. Each pair of vertices $c,d\in \Gamma(a)$ are adjacent (otherwise $b, c, d$ form an induced subgraph $3K_1$), so $\Gamma(a)\cup\{a\}\cong K_{v+1}=K_k$. The same applies to $b$, so $\Gamma\cong 2K_k$, a case dealt with earlier. We may therefore assume that there are induced subgraphs isomorphic to $3K_1$. 

If there are no induced subgraphs isomorphic to $K_2+K_1$, which has just one edge, then, given any edge $ab$, each vertex $c\ne a, b$ is adjacent to $a$ or $b$, so $n\le 2+2(v-1)=2v<2k$, a contradiction. Thus, apart from the cases $\Gamma\cong 2K_k$ and $kK_2$, already dealt with, there are induced subgraphs isomorphic to $P_3$, $K_2+K_1$ and $3K_1$.  \hfill$\square$

\smallskip

(There are three graphs $\Gamma$ satisfying the hypotheses of Lemma~\ref{iso-class} when $n=6$: the graphs $3K_2$ and $2K_3$ each have two isomorphism classes of $3$-vertex induced subgraphs, whereas the $6$-cycle $C_6$ has three. When $n=4$ the only graph is $2K_2$, with two classes of $2$-vertex subgraphs.


\section{$2$-regular groups of automorphisms}\label{2-reggps}

Theorem~\ref{MergedJohnsonCayley} determines those cases where a merged Johnson graph $J$ has a group $G$ of automorphisms acting regularly on the vertices. We will now prove Theorem~\ref{MergedJohnson2reg}, which determines which merged Johnson graphs
$J$ have a group of automorphisms acting $2$-regularly on $J$.

\noindent{\sl Proof.} The structure of the proof follows that of Theorem~\ref{MergedJohnsonCayley}, apart from the use of cohomology of groups at one point. We first deal with some trivial cases.

If $I=\{1,\ldots, k\}$ then $J$ is the complete graph on ${n\choose k}$ vertices, so $G$ can be any transitive permutation group with stabilisers of order $2$, or equivalently, any group of order $2{n\choose k}$ acting on the cosets of a core-free (equivalently non-normal) subgroup $H\cong C_2$, as in conclusion~(5). (Such groups $G$, for example dihedral groups, exist in all cases.) We will therefore assume from now on that $I\ne \{1,\ldots, k\}$.

If $n=2k$ and $I=\{k\}$ or $\{1,\ldots, k-1\}$ then $J$ is either a complete matching $eK_2$ or its complement, where $e={n\choose k}/2$, and hence ${\rm Aut}\,J=S_2\wr S_e$, the maximal imprimitive group with $e$ blocks of size $2$. In this case the groups $G$ acting $2$-regularly are again those of order $2{n\choose k}$ with a non-central subgroup $H\cong C_2$, as in conclusion~(4): one can identify the vertices with the cosets of $H$ and the blocks with the cosets of a subgroup of order $4$ containing $H$ (such subgroups exist since $|G|$ is divisible by $4$). We will therefore assume from now on that if $n=2k$ then $I$ is neither $\{k\}$ nor $\{1,\ldots, k-1\}$.

It is straightforward to check that the groups $G$ in conclusions~(1) and (2) act $2$-regularly on the corresponding graphs $J$, and we will show this for (3) later, under (e). We will now use Theorem~\ref{AutJ} to prove the converse, that there are no other $2$-regular groups of automorphisms than those listed in (1), (2) and (3).

(a) If $G\le S_n$ then $G$ is a $k$-homogeneous permutation group of degree $n$, and the subgroup $G_K$ preserving a $k$-element subset $K$ has order $2$.

If $G$ is not $k$-transitive, it is as listed in Theorem~\ref{k-homog}, cases~(1) to (4). Having even order, $G$ contains an involution, so if $k=2$ then $G$ acts transitively, not only on unordered pairs but also on ordered pairs, that is, it is $2$-transitive, against our assumption. Thus $k\ge 3$, so only cases~(2), (3) and (4) of Theorem~\ref{k-homog} apply. Since $|G_K|=2$ we have $|G|/{n \choose k}=2$. However, by inspection of the individual groups, we see that in case~(2) this ratio is divisible by $3$, in case~(3) it takes the values $1, 3$ and $1$ respectively, and in case~(4) it takes the values $4, 12$ and $4$. 

Thus $G$ is $k$-transitive, so $G_K$ acts on $K$ as $S_k$, and hence $|G_K|\ge k!$; however, $|G_K|=2$, so $k=2$, and $J=J(n,2)_1$ or $J(n,2)_2$. Now $G$, as  $2$-transitive group of degree $n$ and order $n(n-1)$, is sharply $2$-transitive. Such groups were classified by Zassenhaus, and are listed in~\cite[\S 7.6]{DM}: they are the $1$-dimensional affine groups $AGL_1(F)$ over near-fields $F$ (including the finite fields), and their degrees $n$ are the prime powers, giving conclusion~(1).

(b) If $J=J(12,4)_I$ with $I=\{1,3\}$ or $\{2,4\}$ then ${\rm Aut}\,J\cong GO^-_{10}(2)$. Any $2$-regular subgroup $G$ has order $2{12\choose 4}=990$. By the proof of Theorem~\ref{MergedJohnsonCayley} there are no subgroups of order $495$ in $GO^-_{10}(2)$, so $G$ has no subgroups of index $2$; it is therefore contained in $O^-_{10}(2)$, and hence in one of its maximal subgroups. By~\cite{ATLAS}, the only maximal subgroups of  $O^-_{10}(2)$ of order divisible by $990$ are isomorphic to $A_{12}$ or $C_3\times U_5(2)$. If $G\le A_{12}$ then $G$ cannot be transitive on $\{1,\ldots,12\}$ since $|G|$ is not divisible by $12$; having order divisible by $11$, it must have orbits of length $11$ and $1$, so $G\le A_{11}$; however, the transitive groups of degree $11$ are all known, and there is none of order $990$. If  $G\le C_3\times U_5(2)$, the simple group $U_5(2)$ must have a subgroup of order $990$ or $330$; this must be contained in a maximal subgroup of $U_5(2)$, whereas by~\cite{ATLAS} the only maximal subgroups of order divisible by $330$ are isomorphic to $PSL_2(11)$, a simple group of order $660$ with no subgroups of index $2$.

Parts (a) and (b) deal with all cases where $k<(n-1)/2$. We dealt earlier with the cases where $I=\{1,\ldots, k\}$, and where $k=n/2$ and $I=\{k\}$ or $\{1,\ldots, k-1\}$, so we may assume that $n>4$.

(c) Now let $k=n/2$ and $I'\ne I''$, so that ${\rm Aut}\,J=S_n\times S_2$. Cases where $G\le S_n$ are dealt with in (a), so we may assume that $H:=G\cap S_n$ has index $2$ in $G$. If $H$ acts transitively on $N\choose k$ it does so regularly, so $H$ is $k$-homogeneous but not $k$-transitive on $N$. Theorem~\ref{k-homog} now implies that $k<n/2$, against our assumption.

If $H$ is intransitive on $N\choose k$ then, having index $2$ in a transitive group $G$, it must have two orbits of length ${n\choose k}/2$, both $2$-regular. Since $n>4$, Lemma~\ref{rregorbits} implies that  $n=6$, $k=3$, and $H$ is $AGL_1(5)$, acting on $N={\mathbb P}^1(5)$ as the stabiliser of $\infty$ in $PGL_2(5)$. Thus $G$ is $AGL_1(5)\times C_2$, with complementation generating the second factor, as in conclusion~(2).

Now suppose that $k=n/2$ and $I'=I''$; since $I\ne\{k\}$, $\{1,\ldots, k-1\}$, $\{1,\ldots, k\}$, we must have $n>6$. Then ${\rm Aut}\,J=M\rtimes Q$, a semidirect product of a normal subgroup $M\cong S_2^e$ by a complement $Q\cong S_n$: each of the $e={n\choose k}/2$ direct factors $M^{\phi}\cong S_2$ of $M$ transposes the $k$-element sets $K$ and $N\setminus K$ in one equipartition $\phi=\{K,N\setminus K\}$ of $N$, while fixing all other elements of $N\choose k$; the complement $Q$ permutes $k$-element subsets through its natural action on $N$, and permutes the direct factors $M^{\phi}$ of $M$ by conjugation as it permutes the set $\Phi$ of all equipartitions $\phi$ of $N$. Let $\overline G$ denote the image of $G$ under the natural epimorphism ${\rm Aut}\,J\to Q\cong S_n$; since $M$ acts trivially on $\Phi$, and $Q$ acts faithfully, one can regard $\overline G$ as the group induced by $G$ on $\Phi$. Since $G$ acts transitively on $N\choose k$, $\overline G$ is transitive on $\Phi$, so $|\overline G|$ is divisible by $e$. Since $|G|=4e$ it follows that $|G\cap M|$ divides $4$ and $|\overline G|=e$, $2e$ or $4e$.

As a subgroup of $S_n$, $\overline G$ has an action on $N\choose k$, agreeing with $G$ on $\Phi$. If $\overline G$ is transitive on $N\choose k$ then $\overline G$ is $k$-homogeneous on $N$, though not $k$-transitive since $|G|$ is not divisible by ${n\choose k}k!$. However, Theorem~\ref{k-homog} shows that no such group of degree $n$ exists for $k=n/2$, so $\overline G$ is intransitive on $N\choose k$. Since it is transitive on $\Phi$ it must have two orbits on $N\choose k$, transposed by complementation.

If $|\overline G|=e$ or $2e$ then $\overline G$ acts regularly or $2$-regularly on its two orbits in $N\choose k$, so Lemma~\ref{rregorbits} implies that $n=4$ or $6$, a contradiction. Thus $|\overline G|=4e=|G|$, so $\overline G\cong G$ and $|G\cap M|=1$. In this case $\overline G$ acts $4$-regularly on its two orbits in $N\choose k$, so Lemma~\ref{rregorbits} implies that $n=10$ and $\overline G$ is the simple group $S:=PSL_2(8)=SL_2(8)$ of order $504$, fixing a point $p\in N=\{1,\ldots, 10\}$ and acting naturally on $N\setminus\{p\}$, which is identified with ${\mathbb P}^1(8)$; its orbits on ${N\choose k}={N\choose 5}$ consist of those $5$-element subsets containing or not containing $p$, as in conclusion~(3). (See Theorem~\ref{k-homog}(4) for the $4$-homogenous action of $PSL_2(8)$ on ${\mathbb P}^1(8)$.) 

We  still need to demonstrate the existence of such $2$-regular subgroups $G$ of ${\rm Aut}\,J$, and here we will use some concepts from the cohomology of groups: the book~\cite{Bro} by Brown is an excellent reference.
 
The preceding argument shows that if $G$ exists, it is a complement for $M$ in the extension\[E:=M\rtimes S\le M\rtimes Q={\rm Aut}\,J\]
of $M$ by $S$. The conjugacy classes of such complements correspond to the elements of the first cohomology group $H^1(S,M)$ of $S$ with coefficients in $M$, which we will regard as an ${\mathbb F}_2S$-module~\cite[IV.2.3]{Bro}: specifically, the class $[\gamma]\in H^1(S,M)$ containing a cocycle $\gamma:S\to M$ determines the conjugacy class containing a complement $G=\{\gamma(s)s\mid s\in S\}$. Now $M$ is the permutation module over ${\mathbb F}_2$ for $S$ acting on $\Phi$, or equivalently on the cosets of a Klein four-group $V=S_{\phi}\le S$ stabilising some $\phi\in\Phi$. It follows that $M$ is the induced ${\mathbb F}_2S$-module ${\rm Ind}_V^SM^{\phi}$, where the corresponding direct factor $M^{\phi}\;(\cong{\mathbb F}_2)$ of $M$ is regarded as a trivial $1$-dimensional ${\mathbb F}_2V$-module (see~\cite[III.5.5(a)]{Bro}, with coefficients reduced mod~$(2)$). It therefore follows from Shapiro's Lemma~\cite[III.6.2 and III.5.9]{Bro} that restriction from $S$ to $V$ and induction from $V$ to $S$ induce mutually inverse isomorphisms
\[H^1(S,M)\cong H^1(V,M^{\phi}).\]
Since $V$ acts trivially on $M^{\phi}$ we have
\[H^1(V,M^{\phi})\cong {\rm Hom}(V,{\mathbb F}_2)=V^*\cong V\]
\cite[III.1, Exercise 2]{Bro}, so there are four conjugacy classes of complements $G$ for $M$ in $E$, with the conjugates of $S$ (which we will call the {\sl standard\/} complements) corresponding to the zero cocycle.

We now need to determine which of these complements act transitively on $N\choose 5$. Any complement $G$ for $M$ is transitive on $\Phi$, so it is either transitive or intransitive on $N\choose 5$ as the stabiliser $G_K=G\cap E_K$ in $G$ of some $K\in{N\choose 5}$ has index $|G:G_K|=252$ or $126$, or equivalently $G_K$ has order $2$ or $4$. Now the orbit of $M$ containing $K$ consists of $K$ and $N\setminus K$, which are in different orbits of $S$, so $E_K=M_K\rtimes S_K$, where $M_K$ and $S_K$ are the stabilisers of $K$ in $M$ and $S$. Here $S_K$ is the Klein four-group $V=S_{\phi}$ stabilising $\phi=\{K,N\setminus K\}\in\Phi$, while $M_K=\oplus_{\psi\ne \phi}M^{\psi}\cong S_2^{125}$, with its direct factors generated by the transpositions corresponding to the equipartitions $\psi\in\Phi\setminus\{\phi\}$ of $N$. If we regard $M$ as the set of all functions $\Phi\to{\mathbb F}_2$, then $M_K$ consists of those sending $\phi$ to $0$. 

Given a complement $G$ and a corresponding cocycle $\gamma: S\to M$, each element $g\in G$ has the form $g=\gamma(s)s$ for some unique $s\in S$. In particular, we have seen that $g\in G_K$ if and only if $s\in S_K$ and $\gamma(s)\in M_K$, or equivalently, $s\in S_K$ and $\gamma(s)$, regarded as a function $\Phi\to{\mathbb F}_2$, sends $\phi$ to $0$. Now $\gamma$ is induced from a cocycle $\delta:V\to M^{\phi}$, that is, a homomorphism $V\to{\mathbb F}_2$, with $\gamma=0$ if and only if $\delta=0$. Thus $|G_K|=|V\cap \ker \delta|=4$ or $2$ as $G$ is standard or nonstandard. This confirms that the standard complements are intransitive on $J$, and shows that the three classes of nonstandard complements are transitive, as in conclusion~(3).

(d) Finally, to deal with the case $k=(n-1)/2$ we return to the main line of the proof, again following that of Theorem~\ref{MergedJohnsonCayley}. If $I\ne k+1-I$ then ${\rm Aut}\,J=S_n$ as in part~(a), so we may assume that $I=k+1-I$. Thus ${\rm Aut}\,J=S_{n+1}$, with its natural action on $N^*=N\cup\{n+1\}$ inducing actions on the set of equipartitions of $N^*$ and hence on $N\choose k$. Any $2$-regular action of $G$ on ${N\choose k}$ corresponds to a $2$-regular action on equipartitions, which amounts to an action on $N^*$ which is either regular on ${N^*\choose k+1}$ or has two $2$-regular orbits on ${N^*\choose k+1}$ transposed by complementation. In the first case, $G$ is a $(k+1)$-homogenous but not $(k+1)$-transitive group of degree $n+1=2(k+1)$ on $N^*$, contradicting Theorem~\ref{k-homog}. In the second case Lemma~\ref{rregorbits} shows that $n=5$ and $k=2$, so the condition $I=k+1-I$ forces $I=\{1,2\}$, against our assumption that $I\ne \{1,\ldots, k\}$. \hfill$\square$

\smallskip

\noindent{\bf Remark.} It is straightforward to check that in each of conclusions~(1) and (2) there is a unique conjugacy class of $2$-regular subgroups $G$ in ${\rm Aut}\,J$. The following argument shows that the same applies to~(3).

The action of $S$ on ${\mathbb P}^1({\mathbb F}_8)$ extends naturally to $\tilde S:=P\Gamma L_2(8)$ (see Theorem~\ref{k-homog}(4)). This {\bf is} a semidirect product of $S$ by the Galois group $\Gamma:={\rm Gal}\,{\mathbb F}_8/{\mathbb F}_2$ of ${\mathbb F}_8$, a cyclic group of order $3$ generated by the Frobenius automorphism $a\mapsto a^2$. We can take ${\mathbb F}_8={\mathbb F}_2(t)$ where $t^3+t+1=0$. The Sylow $2$-subgroup of the stabiliser $S_{\infty}=AGL_1(8)$ of $\infty$ in $S$ is invariant under $\Gamma$, as is its translation subgroup $V=\{0, t, t^2, t^4\}\cong V_4$ (note that $t^4+t^2+t=0$). The subgroup $\tilde V=V\rtimes\Gamma\cong A_4$ is the stabiliser in $\tilde S$ of the $4$-element subset $V\subset {\mathbb P}^1({\mathbb F}_8)$, so $M$ is the permutation module for $\tilde S$ on the cosets of $\tilde V$. The conjugacy classes of complements for $M$ in $\tilde E:=M\rtimes\tilde S$ therefore correspond to the elements of $H^1(\tilde S,M)$. As before, Shapiro's Lemma gives
\[H^1(\tilde S,M)\cong H^1(\tilde V,N)\cong {\rm Hom}(\tilde V,N).\]
In this case, since $(\tilde V)^{\rm ab}\cong C_3$ we have ${\rm Hom}(\tilde V,N)=0$, so the only complements for $M$ in $\tilde E$ are the conjugates of $\tilde S$. This implies that none of the nonstandard complements $G$ for $M$ in $E$ is invariant under $\Gamma$, since otherwise it would yield a complement $G\rtimes\Gamma$ in $\tilde E$. Now $\Gamma$ permutes the three conjugacy classes of nonstandard complements, either trivially or transitively. Since $|\Gamma|=3$, whereas each conjugacy class has even order, it follows that $\Gamma$ cannot leave any class invariant. Thus the three conjugacy classes of $2$-regular subgroups $G\le E$ are equivalent under $\Gamma$, so they form a single conjugacy class in $\tilde E$. Since the choice of a subgroup $S\cong PSL_2(8)$ in $S_{10}$ is unique up to conjugacy,  it follows that there is a single conjugacy class of $2$-regular subgroups $G$ in ${\rm Aut}\,J$.


\section{Appendix: near-fields and sharply $2$-transitive groups}\label{appx}

The following summary of near-fields and their connection with sharply $2$-transitive permutation groups is adapted to suit our purposes from those given by Cameron in~\cite[\S1.12]{Cam99}, by Dixon and Mortimer in~\cite[\S7.6]{DM}, and by Hall in~\cite[\S20.7]{Hal}.

A {\sl near-field\/} $F$ satisfies all the usual field axioms, apart from possibly commutativity of multiplication $ab=ba$ and the left distributive axiom $a(b+c)=ab+ac$. 
(This concept should not be confused with that of a skew field, or division ring, in which only commutativity of multiplication is relaxed.) A near-field is a field if and only if multiplication is commutative.

In the finite case, near-fields are essentially equivalent to sharply $2$-transitive permutation groups. Firstly, if $F$ is any near-field then the affine group
\[AGL_1(F)=\{t\mapsto ta+b \mid a, b\in F, a\ne 0\}\]
is a sharply $2$-transitive group of transformations of $F$. The translations $t\mapsto t+b$ form a regular normal subgroup $N$, isomorphic to the additive group of $F$, while the transformations $t\mapsto ta\;(a\ne 0)$ form a complement, the stabiliser of $0$, isomorphic to the multiplicative group $F^*=F\setminus\{0\}$. Conversely, any sharply $2$-transitive permutation group $G$ is a primitive Frobenius group, so if it is finite then it has an elementary abelian regular normal subgroup $N$ (the Frobenius kernel), which can be identified with the set permuted. Then the group structures of $N$ and of the stabiliser $G_0$ of the zero element $0\in N$ provide the additive and multiplicative groups of a near-field $F$ with underlying set $N$, such that $G$ acts on $F$ as $AGL_1(F)$. Thus the classification of sharply $2$-transitive finite groups is equivalent to that of finite near-fields; these classifications were achieved by Zassenhaus in~\cite{Zass1, Zass2}.

All but seven of the finite near-fields arise from a construction due to Dickson~\cite{Dic}, in which the multiplicative group of a finite field is `twisted' by field automorphisms to produce a non-commutative multiplication. Here, reversing the order used in~\cite[\S1.12]{Cam99} and~\cite[\S7.6]{DM}, we will describe first the groups and then the near-fields.

If $n$ is a prime power $p^e$ then for each $d$ dividing $e$ the field ${\mathbb F}_n$ has a (unique, cyclic) group of automorphisms $\Gamma$ of order $d$, generated by $\theta:t\mapsto t^q$ where $q=p^{e/d}$. The fixed field of $\Gamma$ is the subfield ${\mathbb F}_q$. Similarly, if 

(a) $d$ divides $n-1=q^d-1$,

then ${\mathbb F}_n^*$ has a (unique, cyclic) subgroup $A$ of index $d$, consisting of its $d$th powers. We will use these two groups to form a group $H$ of semilinear transformations of ${\mathbb F}_n$, fixing $0$ and acting regularly on ${\mathbb F}_n^*$, so that we have a sharply $2$-transitive group
$G=N\rtimes G_0\le A\Gamma L_1({\mathbb F}_n)$ of degree $n$, where $N$ is the translation group and $G_0=H$.

Suppose that

(b) $\{m(i)\mid i=0,\ldots, d-1\}$ is a complete set of residues mod~$(d)$ in $\mathbb Z$,

so that if $\omega$ is any generator for the cyclic group ${\mathbb F}_n^*$ then $\{\omega^{m(i)}\mid i=0,\ldots, d-1\}$ is a set of coset representatives for $A$ in ${\mathbb F}_n^*$. We now let each $g\in A\omega^{m(i)}\subseteq {\mathbb F}_n^*$ induce the semilinear transformation
\[\tau_g:t\mapsto t^{\theta^i}g=t^{q^i}g\]
of ${\mathbb F}_n$, and we define $H:=\{\tau_g\mid g\in {\mathbb F}_n^*\}$. We  need to ensure that $H$ is a group under composition. If $h\in A\omega^{m(j)}$ then (composing from left to right) we have
\[\tau_g \circ \tau_h : t \mapsto (t^{q^i}g)^{q^j}h=t^{q^{i+j}}g^{q^j}h,\]
so for this to have the form $\tau_k$ for some $k\in {\mathbb F}_n^*$ we must ensure that $g^{q^j}h\in A\omega^{m(i+j)}$. This will happen if our chosen residues $m(i)$ satisfy
\[m(i+j)=q^jm(i)+m(j)\]
for all $i$ and $j$. An obvious solution for this identity is to take each
\[m(i)=1+q+q^2+\cdots+q^{i-1}=\frac{q^i-1}{q-1}.\]
Then $H$ is closed under composition, so it is a subgroup of $\Gamma L_1({\mathbb F}_n)$. It fixes $0$, and is transitive on ${\mathbb F}_n^*$ since $\tau_g$ sends $1$ to $g$ for each $g\in {\mathbb F}_n^*$. Since $|H|=|{\mathbb F}_n^*|$ we deduce that $H$ acts regularly on ${\mathbb F}_n^*$, so that
\[G:=\{t\mapsto t\tau_g+b\mid g\in {\mathbb F}_n^*, b\in {\mathbb F}_n\}\]
acts on ${\mathbb F}_n$ as a sharply $2$-transitive subgroup of $A\Gamma L_1({\mathbb F}_n)$.

Note there is an epimorphism $H\to\Gamma, g\mapsto\theta^i$ where $g\in A\omega^{m(i)}$. The kernel is $A$, so (like ${\mathbb F}_n^*$) $H$ is an extension of a normal subgroup $A$ by $\Gamma$. Indeed, $H$ and ${\mathbb F}_n^*$ induce the same group structure on their subgroup $A$, and also on its quotient group, even though $H$, being nonabelian, is not isomorphic to ${\mathbb F}_n^*$ if $d>1$.

Having motivated this definition of $m(i)$, we need to choose $q$ and $d$ carefully so that conditions (a) and (b) are satisfied. Elementary number theory (see~\cite[Theorem~6.4]{Lun} or~\cite{Zass2} for details) shows that the following is a sufficient condition for this:

(c) if $r$ divides $d$, where $r$ is prime or $r=4$, then $r$ divides $q-1$.

The near-field $F$ corresponding to $G$ has the same underlying set and additive structure as ${\mathbb F}_n$, but multiplication (of non-zero elements $g$ and $h$) reflects composition of the corresponding elements of $\tau_g, \tau_h\in H$, so that the product of $g$ and $h$ in $F$ is given by $g\circ h=k=g^{q^j}h$ where $h\in A\omega^{m(j)}$ (and hence the centre of $F$ is ${\mathbb F}_q$). The group $G$ defined above can now be identified with the group $AGL_1(F)$ of affine transformations of this near-field $F$.

Dickson near-fields $F$ of order $n\equiv 3$ mod~$(4)$ appear in Theorem~\ref{MergedJohnsonCayley}(1). In such cases $e$ and hence $d$ are odd, so $H$ has a unique subgroup $H^2$ of index $2$, consisting of those $\tau_g$ such that $g$ is a square in ${\mathbb F}_n^*$. Extending $H^2$ by the translation group gives the unique subgroup $AHL_1(F)$ of index $2$ in $AGL_1(F)$. The involution $-1\in {\mathbb F}_n^*$ is in $A$, so $\tau_{-1}$ acts on $F$ as $t\mapsto -t$; the involution $f:t\mapsto t\tau_{-1}+1$ in $AGL_1(F)$ thus preserves the $2$-element subset $\{0, 1\}$ of $F$, and hence generates its stabiliser. Since $n\equiv 3$ mod~$(4)$, $-1$ is a non-square in ${\mathbb F}_n^*$, so $AHL_1(F)$ does not contain $f$ and hence acts regularly on $2$-element subsets of $F$.

In addition to the Dickson near-fields described above, there are seven exceptional finite near-fields $F$: together with the Dickson near-fields, they appear in Theorem~\ref{MergedJohnson2reg}(1). They have order $n=p^2$, and thus correspond to sharply $2$-transitive groups $G=AGL_1(F)$ of degree $n$, for the primes $p=5$, $7$, $11$ (twice), $23$, $29$ and $59$. In each case, $G$ is a subgroup of $AGL_2(p)$ containing the translation group, and $G_0$ ($\cong F^*$) is a subgroup of $GL_2(p)$ acting regularly on non-zero vectors. When $p=5$, $7$ or $11$ one can take $G_0$ to be the binary tetrahedral, binary octahedral or binary icosahedral group $2T\cong SL_2(3)$, $2O\cong 2.S_4^-$ or $2I\cong SL_2(5)$; for $p=11$ (again), $23$, $29$ or $59$ one can take $G_0=2T\times C_5$, $2O\times C_{11}$, $2I\times C_7$ or $2I\times C_{29}$, with the cyclic direct factor consisting of scalar matrices. 

\bigskip

\centerline{\sc Acknowledgment} 
\vskip 5pt

The second author acknowledges the support by 
the projects VEGA 1/0577/14, VEGA 1/0474/15, NSFC 11371307,
and both authors acknowledge support by the project Mobility - enhancing research, science and
education at the Matej Bel University, ITMS code: 26110230082, under the
Operational Program Education cofinanced by the European Social Fund.

\end{document}